\documentclass[12pt,reqno]{amsart}

\usepackage[mathlines]{lineno}

\usepackage{times}
\usepackage{amsfonts}
\usepackage{amssymb}
\usepackage{latexsym}
\usepackage{xcolor}

\newtheorem{theorem}{Theorem}
\newtheorem{lemma}[theorem]{Lemma}

\newtheorem{corollary}[theorem]{Corollary}

\def\nin{\relax\hbox{$/\kern-.7em{\rm \in\,}$}}

\begin{document}

\begin{center}
\Large{On 
Algebras of Hausdorff Operators on the Real Line}
\end{center}

\bigskip

\begin{center}
E. Liflyand
\end{center}
\begin{center}
Department of Mathematics, Bar-Ilan University, Ramat-Gan 52900, Israel; \quad e-mail: liflyand@gmail.com
\end{center}

\begin{center}
and
\end{center}

\begin{center}
A. Mirotin
\end{center}
\begin{center}
Department of Mathematics and Programming Technologies, Francisk
Skorina Gomel State University, Gomel, Belarus; \quad e-mail: amirotin@yandex.by
\end{center}

\bigskip

{\bf Abstract.} The aim of this work is  to derive a symbol calculus on $L^2(\mathbb{R})$ for one-dimensional Hausdorff operators
in apparently the most general form.

\bigskip


\medskip


\bigskip

\section{Introduction and Preliminaries}

Modern theory and growing interest in Hausdorff operators can be traced back to \cite{Georgakis} and, especially, to \cite{LM}.
One of the simplest properties of such an operator in one of its traditional forms
\begin{equation}\label{ho1}
\mathcal{H}_Kf(x):=\int_{\mathbb{R}}K(u)f(\frac{x}{u})\,du
\end{equation}
is the fact that it is bounded on $L^2(\mathbb{R})$ if $K(u)|u|^{\frac{1}{2}}\in L^1(\mathbb{R})$.

However, there is a feeling that the study of this traditional form is somewhat exhausted. A more general form of the Hausdorff operator
has been recently suggested in \cite{Kuang} and \cite{Kuang1}; we present it in a slightly different manner:

\begin{equation}\label{ho2}
\mathcal{H}_{K,a}f(x):=\int_{\mathbb{R}}K(u)f({x}a(u))\,du,
\end{equation}
where the given functions $K$ and $a$ are measurable and $a(u)\ne 0$ for a.~e. $u\in \mathbb{R}$. Similarly, it is bounded on $L^2(\mathbb{R})$ if
\begin{equation}\label{ho3}
\frac{K(u)}{|a(u)|^{\frac{1}{2}}}\in L^1(\mathbb{R}).
\end{equation}
This follows by straightforward application of generalized Minkowski's inequality.

To make the setting meaningful, we shall assume that $a$ is an odd function, monotonously decreasing to zero on $\mathbb{R}_+$.
It may have a singularity at the origin. In other words, this function mimics the usual $\frac1u$.
  Throughout, we shall also assume that  the condition  (\ref{ho3}) holds.

The aim of this work is  to derive a symbol calculus for one-dimensional Hausdorff operators on $L^2(\mathbb{R})$ of the form (\ref{ho2}).
The notion of a symbol for (generalized) multidimensional Hausdorff operators was introduced in  \cite{faa} and extended in \cite{Forum}.
In our case, the construction of a symbol induces the map
\begin{equation}\label{map}
\mathrm{Smb}:\mathcal{H}_{K,a}\mapsto \Phi, \mathcal{A}_a\to \mathrm{Mat}_2(C_0(\mathbb{R})),
\end{equation}
which is injective and multiplicative. Here
\[
\mathcal{A}_a:=\left\{\mathcal{H}_{K,a}:  \frac{K(u)}{|a(u)|^{\frac{1}{2}}}\in L^1(\mathbb{R})\right\},
\]
$C_0(\mathbb{R})$ stands for the algebra of continuous functions on $\mathbb{R}$ vanishing at infinity,
and $\mathrm{Mat}_2(C_0(\mathbb{R})$ denotes the algebra of matrices of order $2$ with the entries in $C_0(\mathbb{R})$).

It is noteworthy that in some important  cases (see, e.g.,  (\ref{ho1})) the symbol of an one-dimensional Hausdorff operator in a sense of \cite{faa}
is closely related to the notion of a symbol of an integral operator with  homogeneous kernel introduced and studied in \cite{KS}.

There are two main results in this work, Theorems 2.2 and 3.4. We prove and discuss them in the two following sections, respectively.

\bigskip

\section{The algebra $\mathcal{A}_a$}

We begin with a property of the map defined in (\ref{map}).

\begin{lemma}\label{lm1} The map $\mathrm{Smb}: \mathcal{A}_a\to \mathrm{Mat}_2(C_0(\mathbb{R}))$
is an isometry, if we endow the algebra $ \mathrm{Mat}_2(C_0(\mathbb{R}))$ with the norm $\|\Phi\|=\sup_{s\in \mathbb{R}}\|\Phi(s)\|_{op}$.
\end{lemma}

Here  $\|\cdot\|_{op}$ stands for the operator norm of a matrix as the norm of the operator of multiplication by this matrix.

\begin{proof} Let $ M_{\Phi}$ denote the operator of multiplication by the matrix function  $\Phi\in \mathrm{Mat}_2(C_0(\mathbb{R}))$
in the space of vector valued functions $L^2(\mathbb{R}, \mathbb{C}^2)$. It is known from \cite{faa} and  \cite{Forum} that the map
$\mathcal{H}_{K,a}\mapsto M_{\Phi}$ is an isometry (with respect to operator norms)
 if $\Phi=\mathrm{Smb}(\mathcal{H}_{K,a})$. On the other hand,  $\| M_{\Phi}\|=\|\Phi\|$ by \cite[Corollary 3]{Forum}.
\end{proof}

We are now in a position to present our first main result.

\begin{theorem}\label{thm1} The set $\mathcal{A}_a$
is a non-closed  commutative subalgebra of the algebra $\mathcal{L}(L^2(\mathbb{R}))$ of bounded operators on $L^2(\mathbb{R})$  without unit.
\end{theorem}

\begin{proof} Straightforward calculations yield the commutativity of $\mathcal{A}_a$.

Further, the matrix symbol $\mathrm{Smb}(\mathcal{H}_{K,a})=\Phi$ of  an  operator $\mathcal{H}_{K,a}$
can be defined as in \cite{faa} by
\begin{equation}\label{Phi}
\Phi=
\begin{pmatrix} \varphi_+&\varphi_-\\ \varphi_-&\varphi_+
\end{pmatrix},
\end{equation}
where in our case
\begin{equation}\label{phi+}
\varphi_+(s)=\int_{(0,\infty)} \frac{K(u)|u|^{is}}{|a(u)|^{\frac{1}{2}}}\,du=\widehat{K_+}(s),
\end{equation}
\begin{equation}\label{phi-}
\varphi_-(s)=\int_{(-\infty,0)} \frac{K(u)|u|^{is}}{|a(u)|^{\frac{1}{2}}}\,du=\widehat{K_-}(s),
\end{equation}
with $K_{\pm}(t):=\frac{K(\pm e^{-t})e^{-t}}{|a(e^{-t})|^{\frac{1}{2}}}\,\in L^1(\mathbb{R})$  (the ``hat'' stands for the Fourier transform).

Since the map (\ref{map}) is an isometry (and therefore, injective) and multiplicative, to prove that $\mathcal{A}_a$ is an algebra,
it suffices to show that the product of two symbols is also a symbol. More precisely, it suffices to show that
if  $\mathrm{Smb}(\mathcal{H}_{K,a})=\Phi$ and $\mathrm{Smb}(\mathcal{H}_{L,a})=\Psi$,
then $\Phi\Psi=\mathrm{Smb}(\mathcal{H}_{Q,a})$ for some $\mathcal{H}_{Q,a}\in \mathcal{A}_a$.

But
\[
\Phi\Psi=
\begin{pmatrix} \varphi_+&\varphi_-\\ \varphi_-&\varphi_+,
\end{pmatrix}
\begin{pmatrix} \psi_+&\psi_-\\ \psi_-&\psi_+
\end{pmatrix}=
\begin{pmatrix} \varphi_+\psi_+ +\varphi_-\psi_-&\varphi_+\psi_- +\varphi_-\psi_+\\ \varphi_+\psi_- +\varphi_-\psi_+&\varphi_+\psi_+ +\varphi_-\psi_-
\end{pmatrix}
\]
\[
=\begin{pmatrix} ({K_+}\ast {L_+} +{K_-}\ast{L_-})^{\wedge}&({K_+}\ast{L_-} +{K_-}\ast{L_+})^{\wedge}\\ ({K_+}\ast{L_-} +{K_-}\ast{L_+})^{\wedge}&({K_+}\ast{L_+} +{K_-}\ast{L_-})^{\wedge}
\end{pmatrix},
\]
where $\ast$ denotes the convolution in $L^1(\mathbb{R})$.

Defining the functions $Q_{\pm}$ on $\mathbb{R}$ by
\[
Q_+(t):={K_+}\ast {L_+}(t) +{K_-}\ast{L_-}(t),
\]
\[
Q_-(t):={K_+}\ast {L_-}(t) +{K_-}\ast{L_+}(t),
\]
 we obtain
\[
\Phi\Psi=
\begin{pmatrix} \widehat{Q_+}&\widehat{Q_-}\\ \widehat{Q_-}&\widehat{Q_+}
\end{pmatrix}.
\]

Let $Q$ be a function on $\mathbb{R}$ satisfying
\begin{equation}\label{Q}
 Q_{\pm}(t)=\frac{Q(\pm e^{-t})e^{-t}}{|a(e^{-t})|^{\frac{1}{2}}}.
 \end{equation}
   Then $\Phi\Psi=\mathrm{Smb}(\mathcal{H}_{Q,a})$ by the formulas similar to   (\ref{Phi}),  (\ref{phi+}), and (\ref{phi-}).
Since $Q_{\pm}\in L^1(\mathbb{R})$, we have $\frac{Q(u)}{|a(u)|^{\frac{1}{2}}}\in L^1(\mathbb{R})$. Hence, $\mathcal{H}_{Q,a}\in \mathcal{A}_a$.

Choosing a sequence of kernels $K_n$ satisfying (\ref{ho3}), we enjoy the property that the sequence of Fourier transforms $\widehat{K_{n+}}$
converges to a function  from $C_0(\mathbb{R})\setminus W_0(\mathbb{R})$ uniformly on $\mathbb{R}$. Here $ W_0(\mathbb{R})$ denotes the Wiener
algebra of Fourier transforms of functions from $L^1(\mathbb{R})$; for a comprehensive survey, see \cite{LST}. Assume that the sequence of operators $\mathcal{H}_{K_n,a}$
converges to an operator  $\mathcal{H}_{L,a}$ from $ \mathcal{A}_a$ in the operator norm. Then by Lemma \ref{lm1}, the  sequence  of symbols
$\mathrm{Sym}(\mathcal{H}_{K_n,a})$ converges in the norm $\|\cdot\|_{op}$ to $\mathrm{Sym}(\mathcal{H}_{L,a})$ uniformly on $\mathbb{R}$.
But this implies that $\widehat{K_{n+}}$ converges to  $\widehat{L_{+}}\in W_0(\mathbb{R})$  on $\mathbb{R}$, and we arrive at a contradiction.

Finally,  let $\mathcal{H}_{K,a}=I$, the identity operator for some $\mathcal{H}_{K,a}\in \mathcal{A}_a$. Then $\mathrm{Smb}(\mathcal{H}_{K,a})=E_2$
(the unit  matrix of order two) and therefore $\widehat{K_+}(s)=1$, which leads to a contradiction.
This completes the proof.
\end{proof}

\begin{corollary} The algebra $\mathcal{A}_a$ is not Banach.
\end{corollary}

The particular case $a(u)=\frac{1}{u}$ reduces to

{\bf Example 1}. The set
\[
\mathcal{A}:=\{\mathcal{H}_K:K(u)|u|^{\frac{1}{2}}\in L^1(\mathbb{R})\}
\]
is a non-closed  commutative subalgebra of $\mathcal{L}(L^2(\mathbb{R}))$ without unit.

\section{Functions of Hausdorff operators}

In the sequel, let $\sigma(\mathcal{H}_{K,a})$ denote the spectrum of $\mathcal{H}_{K,a}$ in $L^2(\mathbb{R})$.

\begin{theorem} Let $\mathcal{H}_{K,a}\in \mathcal{A}_a$. If a function $F$ is holomorphic in the neighborhood $N$ of the
 set  $\sigma(\mathcal{H}_{K,a})\cup\{0\}$ and $F(0)=0$, then $F(\mathcal{H}_{K,a})\in \mathcal{A}_a$.
\end{theorem}

\begin{proof} Let $\Phi=\mathrm{Smb}(\mathcal{H}_{K,a})$. Then $\mathcal{H}_{K,a}=\mathcal{U}^{-1}M_{\Phi}\,\mathcal{U}$,
where $\mathcal{U}$ ia a unitary operator taking the space $L^2(\mathbb{R}_-)\times L^2(\mathbb{R}_+)$ (which is isomorphic to $L^2(\mathbb{R})$)
into $L^2(\mathbb{R})\times L^2(\mathbb{R})=L^2(\mathbb{R},\mathbb{C}^2)$ \cite{faa}. Moreover, the spectrum of $\mathcal{H}_{K,a}$ equals to
the spectrum of $\Phi$ in the matrix algebra $\mathrm{Mat}_2(C_0(\mathbb{R}))$ \cite{Forum}. This implies (see, e.~g., \cite{DunSw})

\begin{align*}
F(\mathcal{H}_{K,a})&=\frac{1}{2\pi i}\int_{\Gamma} F(\lambda)(\lambda-\mathcal{H}_{K,a})^{-1}d\lambda\\
&=\frac{1}{2\pi i}\int_{\Gamma} F(\lambda)(\lambda-\mathcal{U}^{-1}M_{\Phi}\mathcal{U})^{-1}d\lambda\\
&=\mathcal{U}\left(\frac{1}{2\pi i}\int_{\Gamma} F(\lambda)(\lambda-M_{\Phi})^{-1}d\lambda\right)\mathcal{U}^{-1}\\
&=\mathcal{U}F(M_{\Phi})\mathcal{U}^{-1}=\mathcal{U}M_{F(\Phi)}\mathcal{U}^{-1},
\end{align*}
where  $\Gamma$ is the boundary of any open neighborhood $U$ of the  set  $\sigma(\mathcal{H}_{K,a})\cup\{0\}$ such that $N$ contains its closure. 
To finish the proof, it remains to show that $F(\Phi)$ is the symbol of an operator in $\mathcal{A}_a$. For all regular $\lambda$, we have
\[
(\lambda-\Phi)^{-1}=\frac{1}{\Delta}\begin{pmatrix} \lambda-\varphi_+&-\varphi_-\\ -\varphi_-&\lambda-\varphi_+
 \end{pmatrix},
\]
where $\Delta:=(\lambda-\varphi_+(s))^2-\varphi_-(s)^2\ne 0$ for all $s\in \mathbb{R}$.
Then

\begin{align*}
F(\Phi)&=\frac{1}{2\pi i}\int_{\Gamma} F(\lambda)(\lambda-\Phi)^{-1}d\lambda\\
&=\frac{1}{2\pi i}\int_{\Gamma} F(\lambda)\begin{pmatrix} \frac{\lambda-\varphi_+}{\Delta}&\frac{-\varphi_-}{\Delta}\\
\frac{-\varphi_-}{\Delta}&\frac{\lambda-\varphi_+}{\Delta}   \end{pmatrix}d\lambda\\
&=\begin{pmatrix} \frac{1}{2\pi i}\int_{\Gamma} F(\lambda)\frac{\lambda-\varphi_+}{\Delta}d\lambda&\frac{1}{2\pi i}\int_{\Gamma}
F(\lambda)\frac{-\varphi_-}{\Delta}d\lambda\\
\frac{1}{2\pi i}\int_{\Gamma} F(\lambda)\frac{-\varphi_-}{\Delta}d\lambda&\frac{1}{2\pi i}\int_{\Gamma} F(\lambda)\frac{\lambda-\varphi_+}{\Delta}d\lambda
 \end{pmatrix}\\
 &= \begin{pmatrix} F_1(\varphi_-,\varphi_+)&F_2(\varphi_-,\varphi_+)\\ F_2(\varphi_-,\varphi_+)&F_1(\varphi_-,\varphi_+)
 \end{pmatrix},
\end{align*}
where
\[
F_1(z_1,z_2):=\frac{1}{2\pi i}\int_{\Gamma} F(\lambda)\frac{\lambda-z_2}{(\lambda-z_2)^2-z_1^2}d\lambda,
\]
and
\[
F_2(z_1,z_2):=\frac{-z_1}{2\pi i}\int_{\Gamma} \frac{F(\lambda)}{(\lambda-z_2)^2-z_1^2}d\lambda.
\]

Functions $\varphi_{\pm}=\widehat{K_{\pm}}$ belong to $ W_0(\mathbb{R})$. We are going to prove  that $F_1(\varphi_-,\varphi_+)$ and $F_2(\varphi_-,\varphi_+)$ 
belong to $ W_0(\mathbb{R})$, too. To this end, we employ the functional calculus of several elements of a commutative  Banach algebra with unit (see \cite[\S 13]{GRS} 
or, e.~g., \cite[Ch. III, \S 4]{Gamelin}).  However, the convolution algebra $L^1(\mathbb{R})$ is a  commutative  Banach algebra without unit. Let
\[
\mathcal{V}:=\{\frak{z}=\lambda e +f: \lambda\in \mathbb{C}, f\in L^1(\mathbb{R})\}
\]
be a Banach algebra obtained by the formal adjunction of a unit element $e$ to $ L^1(\mathbb{R})$ (see \cite[\S 16]{GRS}). Each non-zero  complex homomorphism of
 $\mathcal{V}$ is $\psi_s(\lambda e +f)=\widehat{f}(s)$, where $s\in \mathbb{R}$ or $\psi_\infty(\lambda e +f)=\lambda$  \cite[\S 17]{GRS}. In particular, 
 $\frak{z}\in L^1(\mathbb{R})$  if and only if $\psi_\infty(\frak{z})=0.$  We denote by
 $\mathrm{Spec}(\mathcal{V})$ the  Gelfand spectrum (the space of all non-zero  complex homomorphisms) of the algebra $\mathcal{V}$. The map
 $\mathcal{G}\frak{z}(\psi):=\psi(\frak{z})$ ($\psi\in \mathrm{Spec}(\mathcal{V}$)) is  called the Gelfand transform. Then
 \[
 \mathcal{G}(\lambda e +f)=\lambda+\widehat{f}.
 \]
The  joint spectrum of elements $K_{\pm}$ of the commutative  Banach algebra $\mathcal{V}$ is
\begin{align*}
\sigma_{\mathcal{V}}(K_-,K_+)&:=\{(\mathcal{G}K_-)(\psi),  (\mathcal{G}K_+)(\psi):\psi\in \mathrm{Spec}(\mathcal{V})\}\\
&=\{(\widehat{K_-}(s),\widehat{K_+}(s)):s\in \mathbb{R}\}\cup\{(0,0)\}\\
&= \{(\varphi_-(s),\varphi_+(s)):s\in \mathbb{R}\}\cup\{(0,0)\}.
\end{align*}
We claim  that the functions $F_1(z_1,z_2)$ and $F_2((z_1,z_2))$ are holomorphic in a neighborhood of $\sigma_{\mathcal{V}}(K_-,K_+)$. Indeed,
it is known \cite[Theorem 2]{faa} that
$$
\sigma(\mathcal{H}_{K,a})=\mathrm{cl}(\varphi(\mathbb{R})\cup \varphi^*(\mathbb{R})),
$$
where $\varphi=\varphi_++\varphi_-$, $\varphi^*=\varphi_+-\varphi_-$.
It follows that for all $\lambda\in \Gamma$ and $z_1\in \mathrm{cl}(\varphi_-(\mathbb{R}))$,
$z_2\in \mathrm{cl}(\varphi_+(\mathbb{R}))$, we have $(\lambda-z_2)^2-z_1^2\ne 0$ (since  $\lambda\ne z_2\pm z_1$). Therefore,
\[
\min\{|(\lambda-z_2)^2-z_1^2|: \lambda\in \Gamma, (z_1,z_2)\in \sigma_{\mathcal{V}}(K_-,K_+)\}>0,
\]
and thus both functions $F_1$ and $F_2$ are holomorphic on some neighborhood of the  joint spectrum. 
The functional calculus in commutative  Banach algebras implies that there are $Q_{\pm}\in \mathcal{V}$ such that 
$F_{1,2}(\widehat{K_-}, \widehat{K_+})=\mathcal{G}Q_{\pm}$, respectively.

Observing that $F_1(0,0)=F_2(0,0)=0$, we conclude (see, e.~g., \cite[p. 78, Theorem 4.5]{Gamelin}) that 
$\psi_\infty(Q_{\pm})=F_{1,2}(\psi_\infty(K_-), \psi_\infty(K_+))=0$, and so $Q_{\pm}\in L^1(\mathbb{R})$. It follows that
  $F_1(\varphi_-,\varphi_+)=\widehat{Q_-}\in W_0(\mathbb{R})$ and $ F_2(\varphi_-,\varphi_+)=\widehat{Q_+}\in W_0(\mathbb{R})$.
If the function $Q$ on $\mathbb{R}$ is given by  (\ref{Q}), then
   $F(\Phi)=\mathrm{Smb}(\mathcal{H}_{Q,a})$, as desired.
\end{proof}

\section*{Acknowledgments} The authors thank the anonymous referees for their very useful comments and suggestions that improve the presentation.

\medskip

\section*{Funding}

The authors declares that no funds, grants, or other support were received during the preparation of this manuscript.

\medskip

\section*{Competing Interests}

The authors have no relevant financial or non-financial interests to disclose.

\medskip

\section*{Data availability}

This manuscript has no associated data.

\end{document}